\documentclass[a4paper, oneside, 10pt]{amsart}
\usepackage[headheight=12pt, top=2.5cm, bottom=2cm, inner=2.5cm, outer=2.5cm, foot=1cm]{geometry}
\usepackage[utf8]{inputenc}
\usepackage{amsmath,amsthm,amssymb,amsfonts,mathrsfs,microtype}
\usepackage[all, 2cell]{xy} \UseAllTwocells \SilentMatrices

\newtheorem{theorem}{Theorem}
\newtheorem{lemma}[theorem]{Lemma}
\newtheorem{corollary}[theorem]{Corollary}
\newtheorem{proposition}[theorem]{Proposition}
\theoremstyle{definition}
\newtheorem{definition}[theorem]{Definition}



\thanks{This work was supported by the FWF Austrian Science Found (FWF Grant P18779-N13).}

\title[]{A characterization of arithmetical invariants by the monoid of relations}
\author{Andreas Philipp}
\address{Andreas Philipp\\University Graz\\Institute for Mathematics and Scientific Computing\\Heinrichstrasse 36\\8010 Graz\\Austria}
\email{andreas.philipp@uni-graz.at}

\usepackage[scaled=.90]{helvet}
\usepackage{courier}

\begin{document}

\begin{abstract}
The investigation and classification of non-unique factorization phenomena have attracted some interest in recent literature. For finitely generated monoids, S.T. Chapman and P. García-Sánchez, together with several co-authors, derived a method to calculate the catenary and tame degree from the monoid of relations, and they applied this method successfully in the case of numerical monoids. In this paper, we investigate the algebraic structure of this approach. Thereby, we dispense with the restriction to finitely generated monoids and give applications to other invariants of non-unique factorizations, such as the elasticity and the set of distances.
\end{abstract}

\maketitle

\section{Introduction}
An integral domain (more generally, a commutative, cancellative monoid) is called atomic if every non-zero non-unit has a factorization into irreducible elements, and it is called factorial if this factorization is unique up to ordering and associates. Non-unique factorization theory is concerned with the description and classification of non-unique of factorization phenomena in atomic domains. It has its origin in algebraic number theory---the ring of integers of an algebraic number field being atomic but generally not factorial---but in the last decades it became an autonomous theory with many connections to zero-sum theory, commutative ring theory, module theory, and additive combinatorics. We refer to \cite{non-unique} for a recent presentation of the various aspects of the theory.
 
To describe these phenomena, various invariants have been studied in the literature, including the catenary degree, the tame degree, the elasticity, and the set of distances (for some new results, see, e.g., \cite{Ge-Gr-Sc11a} and \cite{B-C-R-S-S10}; for an overview of known results and additional references see, e.g., the monograph \cite{non-unique}; for a statement of the formal definitions, see section \ref{prelim}).

For an integral domain, non-unique factorization phenomena only concern the multiplicative monoid of that domain. Thus we will derive the theory for commutative, cancellative monoids, only.

The monoid of relations associated to a monoid and a certain invariant $\mu(\cdot)$ have been used to study the above mentioned invariants. Investigations of this type started only fairly recently. In \cite{MR1719711}, such investigations were carried out for finitely generated monoids using the results from \cite{MR2254337} and \cite{pre05509372}. In \cite{MR2243561} and \cite{Om}, these results, and expansions thereof, were applied in the investigation of numerical monoids, which are (certain) finitely generated submonoids of the non-negative integers; for a detailed exposition of the theory of numerical monoids and applications, see, e.g., the monograph \cite{pre05623301}.

In the present paper, we focus on the study of the algebraic structure of this method: i.e., the invariant $\mu(\cdot)$, its definition, and the monoid of relations. By this more algebraic-structural approach, we are able to extend the results to not necessarily finitely generated monoids. Furthermore, we address some new aspects. In particular, our investigations include the elasticity and the set of distances.

Moreover, these abstract characterizations, in particular Proposition \ref{pro2}, are used successfully for investigations on the arithmetic of non-principal orders of algebraic number fields in \cite{phil11}. Details however, are too involved to be included here. So the interested reader must be referred to a forthcoming paper dealing that subject.

\section{Preliminaries}
\label{prelim}

In this note, our notation and terminology will be consistent with \cite{non-unique}. Let $\mathbb N$ denote the set of positive integers and let $\mathbb N_0=\mathbb N\uplus\lbrace 0\rbrace$. For integers $n,\,m\in\mathbb N_0$, we set $[n,m]=\lbrace x\in\mathbb N_0\mid n\leq x\leq m\rbrace$. By convention, the supremum of the empty set is zero and we set $\frac{0}{0}=1$. The term ``monoid'' always means a commutative, cancellative semigroup with unit element. We will write all monoids multiplicatively. For a monoid $H$ we denote by $H^\times$ the set of invertible elements of $H$. We call $H$ reduced if $H^\times=\lbrace 1\rbrace$ and call $H_{\mathrm{red}}=H/H^\times$ the reduced monoid associated with $H$. Of course, $H_{\mathrm{red}}$ is always reduced, and the arithmetic of $H$ is determined by $H_{\mathrm{red}}$. Let $H$ be an atomic monoid. We denote by $\mathcal A(H)$ its set of atoms, by $\mathcal A(H_{\mathrm{red}})$ the set of atoms of $H_{\mathrm{red}}$, by $\mathsf Z(H)=\mathcal F(\mathcal A(H_{\mathrm{red}}))$ the free monoid with basis $\mathcal A(H_{\mathrm{red}})$, and by $\pi_H:\mathsf Z(H)\rightarrow H_{\mathrm{red}}$ the unique homomorphism such that $\pi_H|\mathcal A(H_{\mathrm{red}})=\mathrm{id}$. We call $\mathsf Z(H)$ the \emph{factorization monoid} and $\pi_H$ the \emph{factorization homomorphism} of $H$. For $a\in H$, we denote by $\mathsf Z(a)=\pi_H^{-1}(aH^\times)$ the \emph{set of factorizations} of $a$ and denote by $\mathsf L(a)=\lbrace |z|\mid z\in\mathsf Z(a)\rbrace$ the \emph{set of lengths} of $a$. 

In the following, we briefly recall the definitions of all the invariants of non-unique factorization to be dealt with in this paper.
\begin{definition}
Let $H$ be an atomic monoid.
For $a\in H$ we set
\[
\rho(a)=\frac{\sup\mathsf L(a)}{\min\mathsf L(a)},\:\mbox{ and we call }\:
\rho(H)=\sup\lbrace\rho(a)\mid a\in H\rbrace\:\mbox{ the \emph{elasticity} of }H.
\]
\end{definition}
\begin{definition}
Let $H$ be an atomic monoid.
For $a\in H$, the \emph{catenary degree} $\mathsf c(a)$ denotes the smallest $N\in\mathbb N_0\cup\lbrace\infty\rbrace$ with the following property:
\begin{itemize}
\item[] For any two factorizations $z,\,z'\in\mathsf Z(a)$ there exists a finite sequence of factorizations $(z_0,z_1,\ldots,z_k)$ in $\mathsf Z(a)$ such that $z_0=z,\,z_k=z'$, and $\mathsf d(z_{i-1},z_i)\leq N$ for all $i\in[1,k]$.
\end{itemize}
If this is the case, we say that $z$ and $z'$ can be concatenated by an $N$-chain.\\
Also, $\mathsf c(H)=\sup\lbrace\mathsf c(a)\mid a\in H\rbrace$ is called the \emph{catenary degree} of $H$.
\end{definition}

\begin{definition}
\label{def-tame}
Let $H$ be an atomic monoid.
For $a\in H$ and $x\in\mathsf Z(H)$, let $\mathsf t(a,x)$ denote the smallest $N\in\mathbb N_0\cup\lbrace\infty\rbrace$ with the following property:
\begin{itemize}
 \item[] If $\mathsf Z(a)\cap x\mathsf Z(H)\neq\emptyset$ and $z\in\mathsf Z(a)$, then there exists some $z'\in\mathsf Z(a)\cap x\mathsf Z(H)$ such that $\mathsf d(z,z')\leq N$.
\end{itemize}
For subsets $H'\subset H$ and $X\subset\mathsf Z(H)$, we define
\[
 \mathsf t(H',X)=\sup\lbrace\mathsf t(a,x)\mid a\in H',\,x\in X\rbrace,
\]
and we define $\mathsf t(H)=\mathsf t(H,\mathcal A(H_{\mathrm{red}}))$. This is called the \emph{tame degree} of $H$.
\end{definition}

\begin{definition}
Let $\emptyset\neq L\subset\mathbb N_0$ be a non-empty subset and $H$ an atomic monoid.
\begin{enumerate}
 \item A positive integer $d\in\mathbb N$ is called a \emph{distance} of $L$ if there exists some $l\in L$ such that $L\cap[l,l+d]=\lbrace l,l+d\rbrace$. We denote by $\triangle(L)$ the \emph{set of distances} of $L$. Note that $\triangle(L)=\emptyset$ if and only if $|L|=1$.
 \item We call
\[
 \triangle(H)=\bigcup_{a\in H}\triangle(\mathsf L(a))\subset\mathbb N
\]
the \emph{set of distances} of $H$.
\end{enumerate}
\end{definition}

\section{\boldmath $\mu(H)$}

\begin{definition}[$\mathcal R$-relation, cf. {\cite[end of page 3]{Om}}]
Let $H$ be an atomic monoid.
Two elements $z,\,z'\in\mathsf Z(H)$ are \emph{$\mathcal R$-related} if
\begin{itemize}
\item either $z=z'=1$
\item or there exists a finite sequence of factorizations $(z_0,z_1,\ldots,z_k)$ such that $z_0=z,\,z_k=z'$, $\pi_H(z)=\pi_H(z_i)$, and $\gcd(z_{i-1},z_i)\neq 1$ for all $i\in[1,k]$.
\end{itemize}
\end{definition}
We call this sequence an $\mathcal R$-chain concatenating $z$ and $z'$. If two elements $z,\,z'\in\mathsf Z(H)$ are $\mathcal R$-related, we write $z\approx z'$.\\
Since in our general setting the number of factorizations of an element $a\in H$ is not necessarily finite, the number of different $\mathcal R$-equivalence classes of $\mathsf Z(a)$ is potentially infinite, too.

\begin{definition}[$\mu(a),\,\mu(H)$, cf. {\cite[first paragraph, page 4]{Om}}]
Let $H$ be an atomic monoid.
For $a\in H$ let $\mathcal R_a$ denote the \emph{set of $\mathcal R$-equivalence classes of $\mathsf Z(a)$} and, for $\rho\in\mathcal R_a$, let $|\rho|=\min\lbrace|z|\mid z\in\rho\rbrace$. For $a\in H$, we set
\[
\mu(a)=\sup\lbrace |\rho|\mid \rho\in\mathcal R_a\rbrace\leq\sup\mathsf L(a)
\]
and define
\[
\mu(H)=\sup\lbrace\mu(a)\mid a\in H,\,|\mathcal R_a|\geq 2\rbrace.
\]
\end{definition}
Then $\mu(H)=0$ if and only if $|\mathcal R_a|=1$ for all $a\in H$.

\begin{lemma}
\label{lem0}
Let $H$ be an atomic monoid. Then
\[
\mu(H)\geq\mathsf c(H).
\]
\end{lemma}
\begin{proof}
We show that, for all $N\in\mathbb N_0$, all $a\in H$, and all factorizations $z,\,z'\in\mathsf Z(a)$ with $|z|\leq N$ and $|z'|\leq N$, there is a $\mu(H)$-chain from $z$ to $z'$. We proceed by induction on $N$. If $N=0$, then $z=z'=1$ and $\mathsf d(z,z')=0\leq\mu(H)$. Suppose $N\geq 1$ and that, for all $a\in H$ and all $z,\,z'\in\mathsf Z(a)$ with $|z|<N$ and $|z'|<N$, there is a $\mu(H)$-chain from $z$ to $z'$. Now let $a\in H$ and let $z,\,z'\in\mathsf Z(a)$ with $|z|\leq N$ and $|z'|\leq N$. If $z\not\approx z'$, then there are $z'',\,z'''\in\mathsf Z(a)$ such that $z''\approx z$, $z'''\approx z'$, and $z''$ and $z'''$ are minimal in their $\mathcal R$-classes with respect to their lengths. Since $\gcd(z'',z''')=1$, we find $\mathsf d(z'',z''')=\max\lbrace |z''|,|z'''|\rbrace\leq\mu(a)\leq\mu(H)$. Now it remains to show that, for any two factorizations $z,\,z'\in\mathsf Z(a)$ with $z\approx z'$, $|z|\leq N$, and $|z'|\leq N$, there is a $\mu(H)$-chain concatenating them. By definition, there is an $\mathcal R$-chain $z_0,\ldots,z_k$ with $z=z_0$, $z'=z_k$, and $g_i=\gcd(z_{i-1},z_i)\neq 1$ for all $i\in[1,k]$. By the induction hypothesis, there is a $\mu(H)$-chain from $g_i^{-1}z_{i-1}$ to $g_i^{-1}z_i$ for all $i\in[1,k]$, and thus there is a $\mu(H)$-chain from $z_i$ to $z_{i-1}$ for $i\in[1,k]$; thus there is a $\mu(H)$-chain from $z$ to $z'$.
\end{proof}

The following Proposition \ref{pro1} is based on the second part of the proof of \cite[Theorem 3.1]{MR2243561}.
\begin{proposition}
\label{pro1}
Let $H$ be an atomic monoid and $a\in H$ with $|\mathcal{R}_a|\geq 2$. Then
\[
\mathsf c(a)\geq\mu(a).
\]
In particular, $\mathsf c(H)\geq\mu(H)$.
\end{proposition}
\begin{proof}
Let $a\in H$ be such that $|\mathcal R_a|\geq 2$, let $N\in\mathbb N_0$ be such that $\mu(a)\geq N$, and let $z,\,z'\in\mathsf Z(a)$ be such that $z\not\approx z'$, $|z|\geq N$, and $z$ is minimal in its $\mathcal R$-equivalence class with respect to its length. 
There exists a $\mathsf c(a)$-chain of factorizations $z_0,\ldots,z_k$ with $z_0=z$ and $z_k=z'$. As $z\not\approx z'$, there exists some $i\in[1,k]$ minimal such that $z\approx z_j$ for all $j<i$ and $z\not\approx z_i$; then clearly $z_{i-1}\not\approx z_i$, and therefore $\gcd(z_{i-1},z_i)=1$; thus $\mathsf d(z_{i-1},z_i)=\max\lbrace |z_{i-1}|,|z_i|\rbrace$. Since $\mu(a)=|z_0|$, $z_0$ is minimal in its $\mathcal R$-class with respect to its length by definition. Thus we have $|z_0|\leq |z_{i-1}|$.
Then we obtain $N\leq|z|=|z_0|\leq\max\lbrace |z_{i-1}|,|z_i|\rbrace =\mathsf d(z_{i-1},z_i)\leq\mathsf c(a)$. As $N$ was arbitrary, the assertion follows.
\end{proof}

Now we get the result from \cite[Theorem 3.1]{MR2243561} in our slightly more general setup.
\begin{corollary}
\label{cor1}
Let $H$ be an atomic monoid. Then
\[
\mathsf c(H)=\mu(H).
\]
\end{corollary}
\begin{proof}
Clear by Lemma \ref{lem0} and Proposition \ref{pro1}.
\end{proof}

\section{\boldmath The monoid of relations $M_H$}

\begin{definition}
Let $H$ be an atomic monoid. We call
\[
M_H=\lbrace (x,y)\in\mathsf Z(H)\times\mathsf Z(H)\mid\pi_H(x)=\pi_H(y)\rbrace,
\]
the \emph{monoid of relations}.\\
$M_H$, as defined, is the monoid of relations of $H_{\mathrm{red}}$.
\end{definition}

\begin{lemma}
\label{lem1}
Let $H$ be an atomic monoid, $\mathcal P\subset H_{\mathrm{red}}$ be the set of prime elements of $H_{\mathrm{red}}$, and $T=\mathcal A(H_{\mathrm{red}})\setminus \mathcal P$.
\begin{enumerate}
 \item $M_H=\lbrace(qx,qy)\mid q\in\mathcal F(\mathcal P),\,x,y\in\mathcal F(T)\rbrace$ and for all $q\in\mathcal F(\mathcal P)$ and $x,\,y\in\mathsf Z(H)$ we have $(qx,qy)\in M_H$ if and only if $(x,y)\in M_H$. 
 \item \label{lem1.1} The homomorphism $\varphi:M_H\rightarrow\mathcal F(\mathcal P)\times\mathcal F(T)\times\mathcal F(T)$, $\varphi((qx,qy))=(q,x,y)$ with $q\in\mathcal F(\mathcal P)$ and $x,\,y\in\mathcal F(T)$ is a divisor theory.
 \item \label{lem1.2} $M_H$ is a Krull monoid with class group $\mathsf q([T])$, where $\mathsf q([T])$ denotes the quotient group of the monoid generated by the elements in $T$, and the set of all classes containing primes is given by $\lbrace v,v^{-1}\mid v\in T\rbrace\cup\lbrace 1\rbrace$ if $\mathcal P\neq\emptyset$, i.e. $H$ posseses at least one prime element, and by $\lbrace v,v^{-1}\mid v\in T\rbrace$ otherwise.\\
In particular, the set of classes containing primes is finite if and only if $T$ is finite.
\end{enumerate}
\end{lemma}
\begin{proof}\ 
\begin{enumerate}
 \item Obiously, we have $\mathsf Z(H)=\mathcal F(\mathcal P)\times\mathcal F(T)$. Let $(qx,q'y)\in\mathsf Z(H)\times\mathsf Z(H)$ with $q,\,q'\in\mathcal F(\mathcal P)$ and $x,\,y\in\mathcal F(T)$. Then $(qx,q'y)\in M_H$ if and only if $\pi_H(qx)=\pi_H(q'y)$. Since $q,\,q'$ are products of prime elements we find $q=q'$, and thus $\pi_H(x)=\pi_H(y)$.
 \item First we show that $\varphi$ is a divisor homomorphism. Let $(q_1x_1,q_1y_1),\,(q_2x_2,q_2y_2)\in M_H$ be such that $\varphi(q_1x_1,q_1y_1)=(q_1,x_1,y_1)\mid(q_2,x_2,y_2)=\varphi(q_2x_2,q_2y_2)$ in $\mathcal F(\mathcal P)\times\mathcal F(T)\times\mathcal F(T)$. Then there exists $(q,x,y)\in\mathcal F(\mathcal P)\times\mathcal F(T)\times\mathcal F(T)$ such that $(q_1,x_1,y_1)(q,x,y)=(q_2,x_2,y_2)$. Now we apply $\pi_H$ and find
\[
\pi_H(y_1)\pi_H(x)=\pi_H(x_1)\pi_H(x)=\pi_H(x_1x)=\pi_H(x_2)=\pi_H(y_2)=\pi_H(y_1y)=\pi_H(y_1)\pi_H(y).
\]
Thus $\pi_H(x)=\pi_H(y)$, and therefore $(qx,qy)\in M_H$ and $(q_1x_1,q_1y_1)\mid (q_2x_2,q_2y_2)$ in $M_H$.\\
Now we prove that $\varphi$ is a divisor theory. Since $\mathcal F(\mathcal P)\times\mathcal F(T)\times\mathcal F(T)=\mathcal F(U)$ with $U=\lbrace (p,1,1)\mid p\in \mathcal P\rbrace\cup\lbrace (1,t,1),(1,1,t)\mid t\in T\rbrace$, we must show that any elment of $U$ is the greatest common divisor of the image of a finite subset of $M_H$.
Let $(p,1,1)\in U$. Since $\varphi(p,p)=(p,1,1)$, we are done.
Let $u\in\mathcal A(H)$ be not prime such that $(1,uH^\times,1)\in U$. Since $u\in\mathcal A(H)$ is not prime, there are $a,\,b\in H\setminus H^\times$ not divisible by any prime such that $u\mid ab$ but $u\nmid a$ and $u\nmid b$. Now let $z\in\mathsf Z(u^{-1}ab)$, $x\in\mathsf Z(a)$, and $y\in\mathsf Z(b)$ with $uH^\times\nmid xy$. Then we find $(1,uH^\times,1)=\gcd(\varphi(zuH^\times,xy),\varphi(uH^\times,uH^\times))$.
 \item It is clear by part \ref{lem1.1} and \cite[Theorem 2.4.8.1]{non-unique} that $M_H$ is a Krull monoid. Now we compute its class group. We define the map
\[
\phi:\left\lbrace
\begin{array}{ccc}
\mathcal F(\mathcal P)\times\mathcal F(T)\times\mathcal F(T) & \rightarrow & \mathsf q([T]) \\
(q,x,y) & \mapsto & \pi_H(x)(\pi_H(y))^{-1}.
\end{array}
\right.
\]
Obviously, $\phi$ is a well-defined monoid homomorphism and $\phi$ is surjective. By \cite[Proposition 2.5.1.4]{non-unique}, it is sufficient to show that $\phi^{-1}(1)=\varphi(M_H)$ in order to prove that the class group of $M_H$ equals $\mathsf q([T])$. Now let $(q,x,y)\in\mathcal F(\mathcal P)\times\mathcal F(T)\times\mathcal F(T)$ be such that $\phi(q,x,y)=1$. Then we find
\[
\phi(q,x,y)=\pi_H(x)\pi_H(y)^{-1}=1
\quad\Leftrightarrow\quad
\pi_H(x)=\pi_H(y)
\quad\Leftrightarrow\quad
(x,y)\in M_H
\quad\Leftrightarrow\quad
(qx,qy)\in M_H,
\]
and we are done.
For the last part of the proof, we calculate the set of all classes containing prime elements of $\mathcal F(\mathcal P)\times\mathcal F(T)\times\mathcal F(T)$. We have $\mathcal F(\mathcal P)\times\mathcal F(T)\times\mathcal F(T)=\mathcal F(U)$ with $U=\lbrace (p,1,1)\mid p\in \mathcal P\rbrace\cup\lbrace (1,t,1),(1,1,t)\mid t\in T\rbrace$ and find $\lbrace v,v^{-1}\mid v\in T\rbrace\cup\lbrace 1\rbrace$ if $\mathcal P\neq\emptyset$ and $\lbrace v,v^{-1}\mid v\in T\rbrace$ otherwise.
\qedhere
\end{enumerate}
\end{proof}

As we saw in the proof of Lemma \ref{lem1}.\ref{lem1.1} every element of $\mathsf Z(H)\times\mathsf Z(H)$ can be written as greatest common divisor of the image of at most two elements from $M_H$. In the literature, such a Krull monoid is called a $\delta_1$-semigroup with divisor theory; for reference, see \cite{Sk70} and \cite{Sk81}.

\begin{lemma}
\label{lemX}
Let $H$ be an atomic monoid. Then
\[
 \mathcal A(M_H)\subset\lbrace(uH^\times,uH^\times)\mid u\in\mathcal A(H)\rbrace\cup\lbrace(x,y)\in M_H\mid \gcd(x,y)=1\rbrace.
\]
\end{lemma}
\begin{proof}
Let $(x,y)\in\mathcal A(M_H)$ and $z=\gcd(x,y)$. If $z=1$, we are done. Now assume $z\neq 1$. Then $z=u_1\cdot\ldots\cdot u_k$ for some $k\in\mathbb N$ and $u_1,\ldots,u_k\in\mathcal A(H_{\mathrm{red}})$. Now we find $(x,y)=(z,z)(xz^{-1},yz^{-1})=(u_1,u_1)\cdot\ldots\cdot(u_k,u_k)(xz^{-1},yz^{-1})$. If $k\geq 2$, then $(x,y)\notin\mathcal A(M_H)$, a contradiction. If $k=1$, then $(x,y)\in\mathcal A(M_H)$ implies $(xz^{-1},yz^{-1})=(1,1)$, that is, $x=z=y=u_1\in\mathcal A(H_{\mathrm{red}})$.
\end{proof}

\begin{definition}
Let $H$ be an atomic monoid and $M_H$ its monoid of relations. For $(x,y)\in M_H$ and $X\subset M_H$, we set
\[
\widetilde\triangle(x,y)=\big| |x|-|y|\big|\,\mbox{ and }\,\widetilde\triangle(X)=\lbrace\widetilde\triangle(x,y)\mid(x,y)\in X,x\neq y\rbrace.
\]
\end{definition}

Now we can prove something like \cite[Proposition 3.2]{MR2243561} for the catenary degree and a similar result for the elasticity and the set of distances.

\begin{proposition}
\label{pro2}
Let $H$ be an atomic monoid.
\begin{enumerate}
 \item \label{pro2.1} $\mathsf c(H)\leq\sup\lbrace |x|\mid (x,y)\in\mathcal A(M_H)\rbrace$.
 \item \label{pro2.2} $\rho(H)=\sup\left\lbrace\left.\frac{|x|}{|y|}\right|(x,y)\in M_H\right\rbrace=\sup\left\lbrace\left.\frac{|x|}{|y|}\right|(x,y)\in\mathcal A(M_H)\right\rbrace$.
 \item \label{pro2.3} $\triangle(H)\subset\widetilde\triangle(M_H)$, $\max\triangle(H)\leq\max\widetilde\triangle(\mathcal A(M_H))$, and $\min\triangle(H)=\gcd\widetilde\triangle(\mathcal A(M_H))=\min\widetilde\triangle(M_H)$.
\end{enumerate}
\end{proposition}

\begin{proof}
\ 
\begin{enumerate}
 \item Let $a\in H\setminus H^\times$ and let $z,\,z'\in\mathsf Z(a)$ be two different factorizations of $a$. Then, of course, $(z,z')\in M_H$. Thus there are $(x_1,y_1),\ldots,(x_k,y_k)\in\mathcal A(M_H)$ such that $(z,z')=(x_1,y_1)\cdot\ldots\cdot(x_k,y_k)$. Now we can construct the following chain of factorizations: $z=z_0$ and $z_i=z_{i-1}x_i^{-1}y_i$ for $i\in[1,k]$. Then $z_k=z'$. Since $(x_i,y_i)\in\mathcal A(M_H)$, we find $\gcd(x_i,y_i)=1$ or $x_i=y_i=u$ with $u\in\mathcal A(H)\subset\mathsf Z(H)$ by Lemma \ref{lemX}. This implies that either $\mathsf d(z_{i-1},z_i)=\max\lbrace |x_i|,|y_i|\rbrace$ or $\mathsf d(z_{i-1},z_i)=0$. Thus $z$ and $z'$ can be concatenated by a $\max\lbrace |x_i|,|y_i|\mid i\in[1,k]\rbrace$-chain. Since $(x,y)\in\mathcal A(M_H)$ if and only if $(y,x)\in\mathcal A(M_H)$, the assertion follows.
 \item For all $a\in H$, we have that $\mathsf Z(a)\times\mathsf Z(a)=M_H$. Thus we find
\[
\rho(a)=\frac{\sup\mathsf L(a)}{\min\mathsf L(a)}=\sup\left\lbrace\left.\frac{|x|}{|y|}\right|x,y\in\mathsf Z(a)\right\rbrace=\sup\left\lbrace\left.\frac{|x|}{|y|}\right|(x,y)\in\mathsf Z(a)\times\mathsf Z(a)\cap M_H\right\rbrace.
\]
The first equality now follows.
Since $\mathcal A(M_H)\subset M_H$ is a subset, it is clear that
\[
\sup\left\lbrace\left.\frac{|x|}{|y|}\right|(x,y)\in\mathcal A(M_H)\right\rbrace\leq\sup\left\lbrace\left.\frac{|x|}{|y|}\right|(x,y)\in M_H\right\rbrace.
\]
In order to prove equality, we show the following assertion:
\begin{itemize}
 \item[] For all $(x,y)\in M_H$, there is $(x',y')\in\mathcal A(M_H)$ such that $\frac{|x'|}{|y'|}\geq\frac{|x|}{|y|}$.
\end{itemize}
Let $(x,y)\in M_H$ and without loss of generality assume $|x|\geq|y|$. Now there is some $n\in\mathbb N$ and $(x_i,y_i)\in\mathcal A(M_H)$ for all $i\in[1,n]$ such that $(x,y)=(x_1,y_1)\cdot\ldots\cdot(x_n,y_n)$. When we pass to the lengths, we find $|x|=\sum_{i=1}^n|x_i|$ and $|y|=\sum_{i=1}^n|y_i|$. This yields
\[
\frac{|x|}{|y|}\cdot|y|=|x|=\sum_{i=1}^n|x_i|=\sum_{i=1}^n\frac{|x_i|}{|y_i|}|y_i|\leq\max_{i=1}^n\frac{|x_i|}{|y_i|}\sum_{i=1}^n|y_i|=\max_{i=1}^n\frac{|x_i|}{|y_i|}\cdot|y|,
\]
Thus we find
\[
\frac{|x|}{|y|}\leq\max_{i=1}^n\frac{|x_i|}{|y_i|}.
\]
 \item Since, for all $d\in\triangle(H)$, there exist $x,\,y\in\mathsf Z(H)$ such that $|x|-|y|=d$ and $\pi_H(x)=\pi_H(y)$, the inclusion $\triangle(H)\subset\widetilde\triangle(M_H)$ is obvious.\\
Now let $d=\max\triangle(H)$. Then there exists $(x,y)\in M_H$ and $a\in H$ such that $\pi_H(x)=aH^\times$, $|x|-|y|=d$ and $[|y|,|x|]\cap\mathsf L(a)=\lbrace|y|,|x|\rbrace$. There are $k\in\mathbb N$ and $(x_1,y_1),\ldots,(x_k,y_k)\in\mathcal A(M_H)$ such that $(x,y)=(x_1,y_1)\cdot\ldots\cdot(x_k,y_k)$. Since $\sum_{i=1}^k|x_i|=|x|>|y|=\sum_{i=1}^k|y_i|$ there exists $j\in[1,k]$ such that $|x_j|>|y_j|$. Now we show $|x_j|-|y_j|\geq d$. We assume to the contrary $|x_j|-|y_j|<d$. We set $z=y_j\prod_{i=1,i\neq j}^k x_i$. Clearly, $z\in\mathsf Z(a)$ and $|z|=|x|-(|x_j|-|y_j|)\in[|x|-(d-1),|x|-1]\cap\mathsf L(a)$, a contradiction.\\
Let $d=\min\triangle(H)$ and $d'=\min\widetilde\triangle(M_H)$. Since $\triangle(H)\subset\widetilde\triangle(M_H)$, $d'\leq d$ is clear. Now we assume $d'<d$. Then there is $(x,y)\in M_H$ such that $\widetilde\triangle(x,y)=d'<d$, a contradiction.\\ It remains to show that $\min\widetilde\triangle(M_H)=\gcd(\widetilde\triangle(\mathcal A(M_H)))$. We define a map $\overline\triangle:M_H\rightarrow\mathbb Z$ given by $\overline\triangle(x,y)=|x|-|y|$. This is a homomorphism, and $\widetilde\triangle(x,y)=\overline\triangle(x,y)$ for all $(x,y)\in M_H$ such that $|x|\geq|y|$. Since, for all $(x,y)\in M_H$, we have $(y,x)\in M_H$, we find $\overline\triangle(X)=\widetilde\triangle(X)\cup(-\widetilde\triangle(X))\cup\lbrace 0\rbrace$ for all subsets $X\subset M_H$, and thus $\gcd(\overline\triangle(\mathcal A(M_H)))=\gcd(\widetilde\triangle(\mathcal A(M_H)))$. Let now $d'=\gcd(\overline\triangle(\mathcal A(M_H)))\in\mathbb N$ and $d=\min\widetilde\triangle(M_H)$. Then there are $k\in\mathbb N,\,n_1,\ldots,n_k\in\mathbb N$, and $(x_1,y_1),\ldots,(x_k,y_k)\in\mathcal A(M_H)$ such that
\[
d'=\sum_{i=1}^k n_i\overline\triangle(x_i,y_i)=\overline\triangle\left(\prod_{i=1}^k(x_i,y_i)^{n_i}\right),
\]
and since 
\[
0<d'=\left|\sum_{i=1}^k x_i^{n_i}\right|-\left|\sum_{i=1}^ky_i^{n_i}\right|\mbox{, we find } d'=\widetilde\triangle\left(\prod_{i=1}^k(x_i,y_i)^{n_i}\right).
\]
Thus $d'\in\widetilde\triangle(M_H)$ Therefore $d'\geq d$. Since $d'\mid d$, equality follows.
\qedhere
\end{enumerate}
\end{proof}

Next, we mimic the ideas from \cite[page 259 and Theorem 3.2]{MR2243561}.

\begin{definition}
Let $H$ be an atomic monoid.
For $a\in H$, we define
\[
\mathcal A_a(M_H)=\lbrace (x,y)\in\mathcal A(M_H)\mid\pi_H(x)=aH^\times\rbrace
\]
and then set
\[
\nu(H)=\sup\lbrace\mu(a)\mid a\in H,\,\mathcal A_a(M_H)\neq\emptyset,\,|\mathcal R_a|\geq 2\rbrace.
\]
\end{definition}

\begin{proposition}
\label{pro3}
Let $H$ be an atomic monoid. Then
\[
 \mathsf c(H)=\nu(H).
\]
\end{proposition}
\begin{proof}
By Corollary \ref{cor1}, it is sufficient to show that $\mu(H)=\nu(H)$. When we compare the definitions of those two invariants, we see that the only thing we really have to show is that
\[
\lbrace a\in H\mid\mathcal A_a(M_H)\neq\emptyset,\,|\mathcal R_a|\geq 2\rbrace=\lbrace a\in H\mid|\mathcal R_a|\geq 2\rbrace. 
\]
One inclusion is trivial and, for the other one, let $a\in H$ be such that $|\mathcal R_a|\geq 2$, and let $z,\,z'\in\mathsf Z(a)$ be two factorizations of $a$ such that $z\not\approx z'$ and such that both are minimal in their $\mathcal R$-equivalence classes with respect to their lengths. Now assume $(z,z')\notin\mathcal A(M_H)$. Then there are $k\geq 2$ and $(x_1,y_1),\ldots,(x_k,y_k)\in\mathcal A(M_H)$ such that $(z,z')=(x_1,y_1)\cdot\ldots\cdot(x_k,y_k)$. But now we find the following $\mathcal R$-chain from $z$ to $z'$: $z_0=z$ and $z_i=z_{i-1}x_i^{-1}y_i$ for $i\in[1,k]$. Then $z_k=z'$ and $\gcd(z_{i-1},z_i)\neq 1$. Since this is a contradiction we have $(z,z')\in\mathcal A(M_H)$, and thus $(z,z')\in\mathcal A_a(M_H)\neq\emptyset$.
\end{proof}

\begin{theorem}
Let $H$ be an atomic monoid. Then
\[
\mathsf c(H)=\sup\lbrace\mathsf c(a)\mid a\in H,\mathcal A_a(M_H)\neq\emptyset\rbrace.
\]
\end{theorem}
\begin{proof}
Obviously, we have $\mathsf c(H)\geq\sup\lbrace\mathsf c(a)\mid a\in H,\mathcal A_a(M_H)\neq\emptyset\rbrace$. Since, by Proposition \ref{pro1}, $\mathsf c(a)\geq\mu(a)$ for all $a\in H$, we find by Proposition \ref{pro3}, that 
\begin{align*}
\sup\lbrace\mathsf c(a)\mid a\in H,\mathcal A_a(M_H)\neq\emptyset\rbrace
& \geq\sup\lbrace\mu(a)\mid a\in H,\mathcal A_a(M_H)\neq\emptyset\rbrace \\
& \geq\sup\lbrace\mu(a)\mid a\in H,\mathcal A_a(M_H)\neq\emptyset,|\mathcal R_a|\geq 2\rbrace \\
& =\nu(H)=\mathsf c(H).
\qedhere
\end{align*}
\end{proof}

\begin{definition}
Let $H$ be an atomic monoid.
For a non-empty subset $\emptyset\neq Y\subset\mathsf Z(H)$ and a factorization $x\in\mathsf Z(H)$, we set
\[
\mathsf d(x,Y)=\min\lbrace\mathsf d(x,y)\mid y\in Y\rbrace
\]
for the \emph{distance} between $x$ and $Y$.
\end{definition}

\begin{theorem}
\label{theo1}
Let $H$ be an atomic monoid and $u\in\mathcal A(H)$.
\begin{enumerate}
 \item \label{theo1.1} $\mathsf t(H,uH^\times)=\sup\lbrace\mathsf d(x,\mathsf Z(a)\cap uH^\times\mathsf Z(H))\mid a\in uH,\,x\in\mathsf Z(a),\,\mathcal A_a(M_H)\neq\emptyset\rbrace$.
 \item \label{theo1.2} $\mathsf t(H) = \sup\lbrace\mathsf d(x,\mathsf Z(a)\cap uH^\times\mathsf Z(H))\mid a\in uH,\,x\in\mathsf Z(a),\,\mathcal A_a(M_H)\neq\emptyset,\,u\in\mathcal A(H)\rbrace$.
\end{enumerate}
\end{theorem}
\begin{proof}
Without loss of generality, we assume that $H$ is reduced: i.e., $H_{\mathrm{red}}=H$.
\begin{enumerate}
 \item Let $t=\mathsf t(H,u)$ and $d=\sup\lbrace\mathsf d(x,\mathsf Z(a)\cap u\mathsf Z(H))\mid a\in uH,\,x\in\mathsf Z(a),\,\mathcal A_a(M_H)\neq\emptyset\rbrace$. We first prove that $t\leq d$. Assume $a\in uH$. Now we must show that, for all $z\in\mathsf Z(a)$, there exists $z'\in\mathsf Z(a)\cap u\mathsf Z(H)$ such that $\mathsf d(z,z')\leq d$. Let $z\in\mathsf Z(a)$. If $u\mid z$, then we are done by setting $z'=z$, since then $\mathsf d(z,z')=0\leq d$. Now assume that $u\nmid z$. As $a\in uH$, we have $u^{-1}a\in H$, and therefore there is some $\overline z\in\mathsf Z(u^{-1}a)$. Then $u\overline z\in\mathsf Z(a)$ and $u\mid u\overline z$. Since $(z,u\overline z)\in M_H$, there exist $n\in\mathbb N$ and $(x_1,y_1),\ldots,(x_n,y_n)\in\mathcal A(M_H)$ such that $(z,u\overline z)=(x_1,y_1)\cdot\ldots\cdot(x_n,y_n)$. This implies that $(x_i,y_i)\mid (z,u\overline z)$ in $M_H$ for all $i\in[1,n]$ and that there exists some $j\in[1,n]$ such that $u\mid y_j$. Observe that $x_j\mid z$ implies that $u\nmid x_j$. Then $(x_j,y_j)\in\mathcal A_{\pi_H(x_j)}(M_H)$, $\pi_H(x_j)=\pi_H(y_j)\in uH$, and $y_j\in\mathsf Z(\pi_H(x_j))\cap u\mathsf Z(H)$. Now take $y'\in\mathsf Z(\pi_H(x_j))\cap u\mathsf Z(H)$ such that $\mathsf d(x_j,y')=\mathsf d(x_j,\mathsf Z(\pi_H(x_j))\cap u\mathsf Z(H))$. If we now choose $z'=y'zx_j^{-1}$, then $u\mid z'$, $z'\in\mathsf Z(a)$, and $\mathsf d(z,z')=\mathsf d(x_j(zx_j^{-1}),y'zx_j^{-1})=\mathsf d(x_j,y')\leq q$. This proves $t\leq d$.\\
 To prove $t\geq d$, let $z\in\mathsf Z(H)$ with $u\mid\pi_H(z)$ be such that $d=\mathsf d(z,\mathsf Z(\pi_H(z))\cap u\mathsf Z(H))$, and let $y\in\mathsf Z(\pi_H(z))\cap u\mathsf Z(H)$ be such that $d=\mathsf d(z,y)$. Then as $t$ is the tame degree of $H$, there must be an element $x\in\mathsf Z(\pi_H(z))$ with $u\mid x$ and $\mathsf d(z,x)\leq t$ by definition. Now $d=\mathsf d(z,y)=\mathsf d(z,\mathsf Z(\pi_H(z))\cap u\mathsf Z(H))\leq\mathsf d(z,x)\leq t$ follows.
\item Obvious by part \ref{theo1.1} and the very definition of the tame degree.
\qedhere
\end{enumerate}
\end{proof}

Let $H$ be an atomic monoid. Suppose we have a decomposition $\mathcal A(H_{\mathrm{red}})=\biguplus_{i\in I}A_i$, where $I$ is an index set and $A_i\subset\mathcal A(H_{\mathrm{red}})$ for $i\in I$ are non-empty subsets such that
\begin{equation}
\label{one}
\mathcal A(M_H)\cap(\mathcal F(A_i)\times\mathcal F(A_i))=\lbrace (a,a)\mid a\in A_i\rbrace
\mbox{ for all }i\in I.
\end{equation}
Let $a,\,b\in\mathcal A(H_{\mathrm{red}})$ and define an equivalence relation $\simeq$ on $\mathcal A(H_{\mathrm{red}})$ by $a\simeq b$ if $a,\,b\in A_i$ for some $i\in I$. We can extend the canonical projection $\pi_\simeq:\mathcal A(H_{\mathrm{red}})\rightarrow\mathcal A(H_{\mathrm{red}})/\simeq$ to a monoid epimorphism $\overline\pi_\simeq:H_{\mathrm{red}}\rightarrow\overline H:=\left[ [a_i]_\simeq\mid i\in I\right]$ (well defined by (\ref{one})) onto a reduced, atomic monoid, where $a_i\in A_i$ for all $i\in I$. Of course, the possibly most interessting special case is, when $I$ is finite, that is, $\overline H$ is a finitely generated, reduced, atomic monoid.\\
Now we can prove the following result.
\begin{theorem}
 \label{theo2}
Let $H$ and $\overline H$ be as above. Then
\[
\mathsf c(\overline H)\leq\mathsf c(H),
\]
and, if additionally $\pi_\simeq$ induces a homomorphism from $M_H$ onto $M_{\overline H}$, then
\begin{enumerate}
 \item $\mathsf c(H)\leq\max\lbrace |x|\mid (x,y)\in\mathcal A(M_{\overline H})\rbrace$;\\
in particular, if $\mathsf c(\overline H)=\max\lbrace |x|\mid (x,y)\in\mathcal A(M_{\overline H})\rbrace$, then $\mathsf c(H)=\mathsf c(\overline H)$;
 \item $\rho(H)=\rho(\overline H)=\max\left\lbrace\left.\frac{|x|}{|y|}\right|(x,y)\in\mathcal A(M_{\overline H})\right\rbrace$; and
 \item $\mathsf t(\overline H)\leq\mathsf t(H)$.
\end{enumerate}
\end{theorem}
\begin{proof}
Since $\pi_\simeq$ is defined as a map from $\mathcal A(H_{\mathrm{red}})$ onto $\mathcal A(\overline H)$, it trivially extends to $\pi_\simeq:\mathsf Z(H)\rightarrow\mathsf Z(\overline H)$ such that the following diagram commutes:
\[
\xymatrix{
\mathsf Z(H)\ar[d]^{\pi_H}\ar[r]^{\pi_\simeq} & \mathsf Z(\overline H)\ar[d]^{\pi_{\overline H}} \\
H_{\mathrm{red}}\ar[r]^{\overline \pi_\simeq} & \overline H &
}
\]
Now we prove the following two statements.
\begin{enumerate}
 \item[\textbf{A1}] For all $z,\,z'\in\mathsf Z(H)$, $z\approx z'$ implies $\pi_\simeq(z)\approx\pi_\simeq(z')$.
 \item[\textbf{A2}] For all $z\in\mathsf Z(H)$, $|z|=|\pi_\simeq(z)|$.
\end{enumerate}
\begin{proof}[Proof of A1]
Let $z,\,z'\in\mathsf Z(H)$ be such that $\gcd(z,z')\neq 1$. Then $1\neq\pi_\simeq(\gcd(z,z'))\mid\gcd(\pi_\simeq(z),\pi_\simeq(z'))$, and therefore $\gcd(\pi_\simeq(z),\pi_\simeq(z'))\neq 1$. The assertion is now obvious.
\end{proof}
\begin{proof}[Proof of A2]
It is obvious that $|z|=|\pi_\simeq(z)|$ for all $z\in\mathsf Z(H)$.
\end{proof}
By \textbf{A1}, we find $\mu(H)\geq\mu(\overline H)$, and thus, by Corollary \ref{cor1}, we have $\mathsf c(\overline H)=\mu(\overline H)\leq\mu(H)=\mathsf c(H)$. Now we assume that $\pi_\simeq$ induces a homomorphism from $M_H$ onto $M_{\overline H}$.
\begin{enumerate}
 \item By \textbf{A2}, we find $\max\lbrace |x|\mid (x,y)\in\mathcal A(M_H)\rbrace=\max\lbrace |x|\mid (x,y)\in\mathcal A(M_{\overline H})\rbrace$ whence Proposition \ref{pro2} implies that $\mathsf c(H)\leq\max\lbrace |x|\mid (x,y)\in\mathcal A(M_H)\rbrace=\max\lbrace |x|\mid (x,y)\in\mathcal A(M_{\overline H})\rbrace$.
 \item Since $\overline H$ is finitely generated, $M_{\overline H}$ is also finitely generated. Thus we have, by \textbf{A2},
\[
\sup\left\lbrace\left.\frac{|x|}{|y|}\right|(x,y)\in\mathcal A(M_H)\right\rbrace=\sup\left\lbrace\left.\frac{|x|}{|y|}\right|(x,y)\in\mathcal A(M_{\overline H})\right\rbrace=\max\left\lbrace\left.\frac{|x|}{|y|}\right|(x,y)\in\mathcal A(M_{\overline H})\right\rbrace.
\]
Now everything follows by Proposition \ref{pro2}.\ref{pro2.2}.
 \item Obviously, we have $\mathsf d(z,z')\geq\mathsf d(\pi_\simeq(z),\pi_\simeq(z'))$ for all $z,\,z'\in\mathsf Z(H)$. Thus we find $\mathsf t(H)\geq\mathsf t(\overline H)$ by Definition \ref{def-tame}.
\qedhere
\end{enumerate}
\end{proof}

\end{document}